\documentclass[12pt,reqno]{amsart}
\usepackage{amsfonts}
\usepackage{dsfont}
\usepackage{amsfonts}
\usepackage{bbm}
\usepackage{mathrsfs}
\usepackage{color}
\usepackage{caption}
\usepackage{mathrsfs}
\usepackage{amsmath}
\usepackage{enumitem}
\usepackage{diagbox}
\allowdisplaybreaks[4]
\usepackage{amsfonts}
\usepackage{amsthm}
\usepackage{amssymb, amsmath,cite}
\usepackage{bm}
\usepackage{amsthm}
\usepackage{verbatim}
\usepackage{graphicx}
\usepackage{color}
\usepackage{amsfonts}
\usepackage{dsfont}
\usepackage{amsfonts}
\usepackage{bbm}
\usepackage{mathrsfs}
\usepackage{color}
\usepackage{graphicx,amsmath,amsfonts,amssymb,amscd,amsthm}
\usepackage{hyperref}
\usepackage[numbers,sort&compress]{natbib}
\usepackage{enumitem}
\usepackage{xstring}   
\usepackage{etoolbox}
\setenumerate[1]{itemsep=1pt,partopsep=0pt,parsep=\parskip,topsep=1pt}
\setitemize[1]{itemsep=1pt,partopsep=0pt,parsep=\parskip,topsep=1pt}

\newtheorem{theorem}{Theorem}[section]
\newtheorem{lemma}{Lemma}[section]

\newtheorem{remark}{Remark}[section]
\newtheorem{definition}{Definition}[section]

\usepackage{graphicx,amsmath,amsfonts,amssymb,amscd,amsthm}
\usepackage{hyperref}
\usepackage{multirow} 
\usepackage[numbers,sort&compress]{natbib}
\usepackage{array}
\usepackage{booktabs}
\usepackage{multirow}
\usepackage{makecell}
\usepackage{abstract}
\usepackage{graphicx}
\usepackage{float}
\usepackage{pifont}
\hypersetup{colorlinks,linkcolor=blue,citecolor=blue}
\newtheoremstyle{kai}
{3pt} {3pt} {} {} {\bfseries} {.} {.5em} {}
\def\EquationsBySection{\def\theequation
{\thesection.\arabic{equation}}
\@addtoreset{equation}{section}}

\newcommand\old[1]{}

 \newcommand{\Bp}{\begin{proof}}
 \newcommand{\Ep}{\end{proof}}
\renewcommand{\theequation}{\thesection.\arabic{equation}}

 \newcommand{\beq}{\begin{equation}}
\newcommand{\eeq}{\end{equation}}
\newfloat{figtab}{htb}{fgtb}
\makeatletter
  \newcommand\figcaption{\def\@captype{figure}\caption}
  \newcommand\tabcaption{\def\@captype{table}\caption}
\makeatother
\DeclareMathOperator{\dive}{div}
\numberwithin{equation}{section}
\addcontentsline{toc}{section}{Acknowledgement}
\begin{document}
\author{Xin Xu}
\address{Xin Xu, School of Mathematical Sciences, South China Normal University, Guangdong 510631, China.}
\email{xuxin1994pkq@163.com}

\author{Kexin Zhang}
\address{Kexin Zhang, Chern Institute of Mathematics, Nankai University, Tianjin 300071, China.}
\email{kxzmath@163.com}


\title
{Asymptotics of the principal eigenvalue of an elliptic operator on closed and orientable  Riemannian manifolds}


\date{}
\maketitle

\begin{abstract}
This paper investigates the asymptotic behavior of the principal eigenvalue $\lambda(s)$, as $s\to+\infty$, for the following elliptic eigenvalue problem
\begin{equation*}\label{E}
  -\Delta_{M}u-s\langle  \nabla_M f, \nabla_M u\rangle_g +c u=\lambda(s)u,
\end{equation*}
defined on an orientable and closed  Riemannian manifold $(M,g)$. Assuming $f$ is a Morse function defined on $M$, we find that the limit $\lim\limits_{s\to+\infty} \lambda(s)$ is determined by the minimum value of the function $c$ over the set of the maximum points of $f$, a result that is independent of the  curvature of manifold.
\end{abstract}

\vskip20pt
\section{Introduction}
In this paper, we are concerned with the  limit of the principal eigenvalue $\lambda(s)$, as $s\to+\infty$, corresponding to the following elliptic eigenvalue problem
\begin{equation}\label{E}
  -\Delta_{M}u-s\langle \nabla_M f,\nabla_M u\rangle_g +c u=\lambda(s)u,
\end{equation}
on an orientable and closed, i.e., compact and without boundary, Riemannian manifold $(M,g)$ of dimension $n$, associated with a Riemannian metric $g$.  Here, $\nabla_M$ is the gradient and $\Delta_M $ is the Laplacian on the manifold $M$. We suppose that the parameter $s$ is positive and the functions $f$ and $c$ are smooth on $M$. In fact, the drifted Laplacian
$\Delta_f=\Delta_M+\langle\nabla_M f,\nabla_M\cdot\rangle_g$
 has many favorable properties. See the appendix for details.

 It is widely recognized that the eigenvalue problem \eqref{E} has an eigenvalue $\lambda(s)$ whose real part is smaller than those of all other eigenvalues,
see \cite{RGA}. This eigenvalue is commonly called the \textit{principal eigenvalue}, and it is both real and simple. Furthermore, this eigenvalue corresponds to an eigenfunction, known as the \textit{principal eigenfunction}, which can be chosen to be positive on $M$. 

To first ground our geometric intuition in a familiar setting, we begin by considering the eigenvalue problem \eqref{E} on a bounded domain $\Omega$ in Euclidean space \(\mathbb{R}^n\) equipped with the flat metric.
Meanwhile, $\Delta_M$ and $\nabla_M$ in \eqref{E} become the familiar Laplacian $\Delta $ and gradient $\nabla $. Since the metric is flat, the inner-product term with respect to $g$ in the equation \eqref{E}  can be written as the standard inner product in Euclidean space.
 The corresponding linear elliptic operator eigenvalue problem:
 \begin{equation}\label{V}
  \left\{
\begin{array}{l}
  -D\Delta u(x)-s\mathbf{v} \cdot\nabla u(x)+V(x)u(x)=\lambda(s)u(x),\quad x\in \Omega,\\
 \mathcal{ B} u=0,\quad x\in\partial\Omega
  \end{array}\right.
\end{equation}
has been extensively studied. Here, $D$ denotes the diffusion rate, $\mathbf{v}(x)$ is a vector field in $\Omega$, $V(x)$ is a continuous function in $\Omega$, and
$\mathcal{ B} u=0$ represents the boundary condition.

When \(\mathcal{ B} u=0\) denotes the \textit{Dirichlet boundary condition}, the asymptotic behavior of the principal eigenvalue $\lambda(s)$ as $D \to 0$ was extensively studied in \cite{DEF73, FA7273, WENTZEL, W75} under various conditions on $\mathbf{v}$. Following these, the works \cite{DF78, FS97} analyzed the asymptotic properties of the principal eigenfunctions.
When \(\mathcal{ B}u =0 \) denotes the \textit{Neumann boundary condition}, in the case that $\dive \mathbf{v}=0$, Berestycki et. al. in \cite{Bcmp} proved that $\lambda(s)$ converges to the minimum value of a specific functional as $s\to+\infty$.
In the case that $\mathbf{v}=\nabla f$ for some $f$ which is non-degenerate,
Chen and Lou in Theorem 1.1 of \cite{CL} established that
$$
\lim_{s\to+\infty}\lambda(s)=\min_{x\in\mathcal{M}}V(x),
$$
where $\mathcal{M}$ is the set of points of local maximum of $f$ on $\Omega$.
In the case that $\mathbf{v}$ is a  general advection term in dimension two, under some necessary restrictions,
the asymptotic behavior of the principal eigenvalue $\lambda(s)$,  as  \( s \) tends to infinity, has been well studied, see \cite{LLZ}.
Additional studies on the influence of advection terms on the principal eigenvalues of second-order elliptic or parabolic operators   are available in the references, such as \cite{HNR2011,LLPZd,LLPZx,PZhao,PZ,N09,N10}.

Our motivation for studying the eigenvalue problem \eqref{E} on manifolds stems from its natural emergence in several geophysical and biological contexts.

In geophysical fluid dynamics, the troposphere ($\approx 10\,\mathrm{km}$) on Earth is four orders of magnitude thinner than the Earth's radius ($\approx 6\,400\,\mathrm{km}$), so large-scale winds are naturally described by functions on the sphere $S^{2}$ \cite{AAA,AM}.Such spherical geometry is also a canonical setting for investigating reaction-diffusion processes and complex pattern formations, where the global constraints of the manifold play a decisive role \cite{Varea1999}.

 Similarly, in mathematical biology, many species inhabit domains that are intrinsically curved and closed. For instance, migratory birds, marine plankton, or the global spread of pathogens occupy habitats that cover the entire Earth's surface; these individuals cannot leave this closed curved space.
 In these biological contexts, the curvature of the underlying surface is known to exert a fundamental influence on the spectral properties of diffusion operators, which in turn dictates the spatial distribution and persistence of species \cite{Plaza2004}.

 In both examples, the natural domain is a closed, orientable Riemannian manifold, prompting our inquiry into the asymptotic behavior of the principal eigenvalue
$\lambda(s)$ in this geometric setting.


In this paper, we study the asymptotic behavior of the principal $\lambda(s)$ eigenvalue as $s\to+\infty$ on the closed and orientable Riemannian manifolds and our main result can be read as follows.
\begin{theorem}\label{mmm}
Asumme that $M$ is a closed and orientable Riemannian manifold, and the function $f: M\rightarrow \mathbb{R}$ is a Morse function. Then, for the principal eigenvalue $\lambda(s)$ of the eigenvalue problem \eqref{E}, we have
$$
\lim_{s\to+\infty}\lambda(s)= \min_{p\in\mathcal{M}}c(p),
$$
where $\mathcal{M}$ represents the maximum points set of the function $f$ on $M$.
\end{theorem}

The precise definition and the properties of Morse function on a manifold can be found in the next section, see Definition \ref{mf} and Lemma \ref{morsel}.

\begin{remark}
Compared to problem \eqref{V} on a bounded Euclidean domain, problem \eqref{E} is posed on a closed manifold and hence requires no boundary conditions.
This result demonstrates that the asymptotic behavior of the principal eigenvalue $\lambda(s)$ on closed orientable Riemannian manifolds exhibit a structure analogous to that of Theorem 1.1 in Chen and Lou \cite{CL}. We show that, under large drift, the limit of $\lambda(s)$ on bounded domains $\Omega\in \mathbb{R}^n$ extends naturally to general closed orientable Riemannian manifolds $(M,g)$, provided that $f$ is a Morse function.


\end{remark}

\begin{remark}
An interesting observation is that, although the problem is set on a manifold, the limiting value $\min_{p\in\mathcal{M}}c(p)$ contains no explicit curvature terms. There are two reasons for this.
First, for a fixed Morse function $f$, in the large drift regime $s\to\infty$, the limit depends only on the local behavior of $f$ near its maximum points.
Second, while the  geometry and topology of manifolds restrict which Morse functions \(f\) are admissible, once \(f\) is fixed, the limit itself does not involve curvature. In other words, the geometry affects only the admissibility of \(f\), not the explicit form of the limit.
\end{remark}

The rest of this paper is organized as follows. Section \ref{pre} provides some preliminaries, including some basic definitions of functions  (Morse function) on the manifold $M$, and  the classification of the critical points of the Morse functions. In Section \ref{pf}, we give the proof of Theorem \ref{mmm} by estimating the upper and lower bound of the principal eigenvalue $\lambda(s)$. Section \ref{app} is the appendix which provides the variational characterization of the principal eigenvalue $\lambda(s)$ of \eqref{E}.

\section{Preliminaries}\label{pre}

Let $(M,g)$ be a closed, orientable Riemannian manifold.
For a local chart $(U,\varphi)$ we denote
\[
\varphi(p)=\bigl(x^{1}(p),\dots,x^{n}(p)\bigr)\in\mathbb{R}^{n},\quad p\in U,
\]
by the coordinate functions $x^{i}:U\to\mathbb{R}$.
The metric $g$ is then represented by a smooth, symmetric, positive-definite matrix $(g_{ij}(x))_{n\times n}$, with $x=(x^{1},\dots,x^{n})=\varphi(p)\in \varphi(U)$.

\begin{definition}[See \cite{MT}]\label{cp}
 Assume that $f: M \rightarrow \mathbb{R}$ is smooth. A critical point  of $f$ is a point $p\in M$, satisfying the differential $df_p = 0$. The critical point $p$, is called non-degenerate, if the Hessian matrix \( H_f(p) \) is non-singular.
\end{definition}

According to Definition \ref{cp}, we have $df_p = 0$, which implies that
\[
\left.\frac{\partial}{\partial x^{i}}\right\vert_{p} f
=\left.\frac{\partial}{\partial x^{i}}\right\vert_{\varphi(p)}(f\circ\varphi^{-1})
=\frac{\partial \hat{f}}{\partial x^{i}}(\varphi(p))=0,\quad \text{for all }i=1,\dots,n,
\]
in the local coordinate $(U,\varphi)$ by \cite{jmlee}, where $\hat{f}=f\circ\varphi^{-1}$.
In what follows, we adopt the simplified notation $\frac{\partial \hat{f}}{\partial x^{i}}(\varphi(p))$ to denote $\left.\frac{\partial}{\partial x^{i}}\right\vert_{p} f$. So we obtain the local coordinate representation of $|\nabla_M f_p|^2$:
\[
|\nabla_M f_p|^2=g^{ij}\frac{\partial \hat{f}}{\partial x^{i}}\frac{\partial \hat{f}}{\partial x^{j}}(\varphi(p)), \quad \text{(summation over }i,j \text{ by Einstein convention),}
\]
where $g^{ij}$ is the $(i,j)$-component of the inverse matrix of the Riemannian metric $(g_{ij}(x))$. So $df_p = 0$ implies that $|\nabla_M f_p|=0$. Moreover, if the critical point $p$ is non-degenerate, we have
\[
\det\left(\frac{\partial^{2} \hat{f}}{\partial x^{i}\partial x^{j}}(\varphi(p))\right)\neq 0.
\]

Since in the local coordinate $(U,\varphi)$, the Riemannian metric $(g_{ij}(x))$, where $x=\varphi(p)$, is represented by a real symmetric, smooth and positive definite matrix and so is the inverse metric $(g^{ij}(x))$, then on the compact set $\overline{\varphi(U)}$, for any vector $\xi\in\mathbb{R}^{n}$, there exist two positive constants $e_1<e_2$ independent of $x$, such that
\begin{equation}\label{g12}
e_1|\xi|^2\le g^{ij}(x)\xi_i\xi_j\le e_2|\xi|^2,\quad x\in\overline{\varphi(U)}.
\end{equation}
By the property of the Riemannian metric, we can also obtain that there exist two positive constants $e_3,e_4$ independent of $x$, such that for $x\in\varphi(U)$,
\begin{equation}\label{g34}
e_3\le\sqrt{\det g_{ij}(x)}\le e_4.
\end{equation}

On smooth manifolds, Morse functions, those with exclusively non-degenerate critical points, are characterized by the following fundamental definition and lemma.
\begin{definition}[Morse Function \cite{MT}]\label{mf}
A smooth function \( f: M \rightarrow \mathbb{R} \) is called a  Morse function if every critical point of \( f \) is non-degenerate.
\end{definition}
\begin{lemma}[Morse's Lemma \cite{MT}]\label{morsel}
Let \( M \) be a smooth manifold and \( f: M \to \mathbb{R} \) be a Morse function. Suppose that \( p \in M \) is a non-degenerate critical point of \( f \),  then there exist a local coordinates \((x^1,...,x^n) \) around \( p \), such that in these coordinates, the function \( f \) can be written as:
\[
f(x)=f(p)-\frac{1}{2}\sum_{i=1}^{k} (x^i)^2+\frac{1}{2}\sum_{i=k+1}^{n} (x^i)^2,
\]
where $k=0,1,...,n$.
\end{lemma}
The critical point $p$ is classified as:
\begin{itemize}
\item A maximal point when $k=n$:
\[
f(x)=f(p)-\frac{1}{2}\sum_{i=1}^{n} (x^i)^2.
\]
\item A minimal point when $k=0$:
\[
f(x)=f(p)+\frac{1}{2}\sum_{i=1}^{n} (x^i)^2.
\]
\item A saddle point when $1\leq k\leq n-1$:
\[
f(x)=f(p)-\frac{1}{2}\sum_{i=1}^{k} (x^i)^2+\frac{1}{2}\sum_{i=k+1}^{n} (x^i)^2.
\]
\end{itemize}

\textbf{Variational.} The principal eigenvalue $\lambda(s)$ is characterized by
\begin{eqnarray}\label{voe}
\lambda(s)
=\min_{u\in H^1(M,\mu)\atop u\not\equiv 0}\frac{\int_{M} e^{sf }\left(|\nabla_M u|^2+c |u|^2\right)\mathrm{dV_g} }{\int_{M}  e^{ sf}u^2\mathrm{dV_g}  },
\end{eqnarray}
where $\mathrm{dV_g}$ is the volume element induced by the metric and the derivation of \eqref{voe} is in the appendix.

We denote by the positive function $v(s,\cdot)$ the positive principal eigenfunction corresponding to the principal eigenvalue $\lambda(s)$ of problem \eqref{E}, satisfying
$$
\int_{M}  e^{ sf}v^2(s,\cdot)\mathrm{dV_g} =1.
$$
Let $w=e^{\frac{sf}{2} }v$. Then
$$\int_M w^2(s,\cdot) \mathrm{dV_g}=1,
$$
so that $w^2(s,\cdot) \mathrm{d}V_g$ defines a family of probability measures, Since the manifold $M$ is compact, this family is tight. According to the Prokhorov's theorem \cite{yan}, there exists a subsequence $\{s_m\}_{m=1}^{+\infty}$ with $s_m\to\infty$ and a probability measure $\mu$ such that
\begin{equation}\label{weak}
\lim_{m\to+\infty}\int_M w^2(s_m,\cdot)\zeta\mathrm{dV_g}=\int_M \zeta\mu(\mathrm{dV_g}), \quad \forall\zeta\in C(M),
\end{equation}
where $C(M)$ is the set of continuous functions on the manifold $M$.

\section{Proof of Theorem \ref{mmm}}\label{pf}
In view of \eqref{voe}, it is straightforward to establish that
$$
c_*\le\lambda(s)\le c^*,
$$
where $c_*=\min_{p\in M}c(p)$ and  $c^*=\max_{p\in M}c(p)$. It is easy to see that if the function $c(\cdot)$ is constant, i.e., $c(p)\equiv c^*=c_*$, we have $\lambda(s)=c^*$ for all $s$. Therefore, without loss of generality, we assume $c^*>c_*>0$ in the remainder of the proof.

\subsection{Upper bound}
In this subsection, we derive the following lemma to give the upper bound estimate for the principal eigenvalue $\lambda(s)$.

\begin{lemma}\label{up}
For the principal eigenvalue $\lambda(s)$ of the eigenvalue problem \eqref{E}, we have the following upper bound estimate
$$
\limsup_{s\to+\infty}\lambda(s)\le \min_{p\in\mathcal{M}}c(p).
$$
\end{lemma}
\begin{proof}
For any fixed point $p^*\in \mathcal{M} $, there exists a smooth local coordinate chart $(U, \varphi)$. 
Select a sufficiently small $R>0$ such that the ball \( B_{3R}(\varphi(p^*)) \subset\subset \varphi(U)\), where $B_{3R}(\varphi(p^*))$ denotes the ball of radius \( 3R \) centered at \( \varphi(p^*) \). We then define the function $\phi_R:\ \mathbb{R}^n\rightarrow\mathbb{R}$ as follow.
\begin{equation}
\phi_R(x)=\phi_{R}(x^1,x^2,...,x^n)=\begin{cases}
1, &   x \in B_{2R}(\varphi(p^*)), \\
\frac{3R-|x|}{R}, &  x\in \overline{B_{3R}(\varphi(p^*))}\setminus B_{2R}(\varphi(p^*)),\\
0,&  x\in \varphi(U)\setminus \overline{B_{3R}(\varphi(p^*))}.
\end{cases}
\end{equation}
Next, we define the test function $u_R: M\rightarrow \mathbb{R}$ as follow.
\begin{equation}\label{test}
u_R=\begin{cases}
\phi_R\circ\varphi, &   p \in U, \\
0, &  p\in U^c.
\end{cases}
\end{equation}
Thanks to the local coordinate chart $(U, \varphi)$, we have
\begin{align}\label{grad}
&|\nabla_M u_R|^2
=g^{ij}\frac{\partial \hat{u}_R}{\partial x^j}\frac{\partial \hat{u}_R}{\partial x^i}
=g^{ij}\frac{\partial \phi_R}{\partial x^j}\frac{\partial \phi_R}{\partial x^i},
\end{align}
where $\hat{u}_R=u_R \circ\varphi^{-1}$ and the detailed calculation process can be found in \cite{jmlee} or see Section \ref{pre}.
To give the upper bound estimate of $\lambda(s)$, we insert the test function into \eqref{voe}. Thanks to the local coordinate expression \eqref{grad} of the gradient of $u_R$, and the basic coordinate transformation formula, we obtain
\begin{eqnarray}
\lambda(s)
&\le&
\frac{\int_{U} e^{sf }\left(|\nabla_M u_R|^2+c |u_R|^2\right)\mathrm{dV_g}}{\int_{U}  e^{ sf}u_R^2\mathrm{dV_g}  }\nonumber\\
&=&\frac{\int_{\varphi(U)} e^{s\hat{f} } g^{ij}\frac{\partial \phi_R}{\partial x^i}\frac{\partial \phi_R}{\partial x^j}\sqrt{\det g_{ij} }\mathrm{d}x }{\int_{\varphi(U)}  e^{ s\hat{f}}\phi_R^2\sqrt{\det g_{ij}} \mathrm{d}x  }+\frac{\int_{\varphi(U)}  e^{ s\hat{f}}\hat{c}\phi_R^2\sqrt{\det g_{ij}} \mathrm{d}x  }{\int_{\varphi(U)}  e^{ s\hat{f}} \phi_R^2\sqrt{\det g_{ij}}\mathrm{d}x  }\nonumber\\
&\triangleq&I+J,
\end{eqnarray}
where
\begin{equation}
I=\frac{\int_{\varphi(U)} e^{s\hat{f} } g^{ij}\frac{\partial \phi_R}{\partial x^i}\frac{\partial \phi_R}{\partial x^j}\sqrt{\det g_{ij}}\mathrm{d}x }{\int_{\varphi(U)}  e^{ s\hat{f}}\phi_R^2\sqrt{\det g_{ij}} \mathrm{d}x  },\quad J=\frac{\int_{\varphi(U)}  e^{ s\hat{f}}\hat{c}\phi_R^2\sqrt{\det g_{ij}} \mathrm{d}x  }{\int_{\varphi(U)}  e^{ s\hat{f}} \phi_R^2\sqrt{\det g_{ij}}\mathrm{d}x  },
\end{equation}
and $\hat{f}=f\circ \varphi^{-1}$, $\hat{c}=c\circ \varphi^{-1}$ are the local coordinate representations of $f$ and $c$, respectively.

We select the fixed point $p$ as a maximum point of $f$ on $M$. According to the Morse Lemma, in the local coordinate $(U, \varphi)$ containing  the local maximum point  $p$ of the Morse function $f$, we have
\begin{equation}\label{mp}
\hat{f}(x)=f\circ \varphi^{-1}(x)=f(p^*)-\sum_{i=1}^n(x^i)^2, \quad x\in \varphi(U).
\end{equation}

For simplicity, we use $C$ represents the positive constant which may change from line to line in the following proof.
Then, for term $I$, by direct calculation, we obtain
\begin{eqnarray}\label{I}
I&=&\frac{\int_{\varphi(U)} e^{s\hat{f} } g^{ij}\frac{\partial \phi_R}{\partial x^i}\frac{\partial \phi_R}{\partial x^j}\sqrt{\det g_{ij}}\mathrm{d}x }{\int_{\varphi(U)}  e^{ s\hat{f}}\phi_R^2\sqrt{\det g_{ij}} \mathrm{d}x  }
\le\frac{C\int_{\varphi(U)} e^{s\hat{f} } |\nabla \phi_R|^2\mathrm{d}x }{ \int_{\varphi(U)}  e^{ s\hat{f}}\phi_R^2  \mathrm{d}x } \nonumber\\
&\le&\frac{C\int_{ \overline{B_{3R}(\varphi(p^*))}\setminus B_{2R}(\varphi(p^*))} e^{s(f(p^*)-\sum_{i=1}^n (x^i)^2) } \mathrm{d}x }{R^2\int_{B_{R}(\varphi(p^*))}  e^{ s(f(p^*)-\sum_{i=1}^n (x^i)^2)}  \mathrm{d}x } \nonumber\\
&\le&\frac{C|B_{3R} | e^{s(f(p^*)-4R^2) }   }{R^2 |B_{R} |  e^{ s(f(p^*)-R^2)}   }
=\frac{C|B_{3R} |    }{R^2 |B_{R} |     }e^{-3sR^2}, \nonumber\\
\end{eqnarray}
where we use \eqref{g12}, \eqref{g34} in the first inequality and \eqref{mp} in the second inequality.

For the second  term, it is obvious to verify that
\begin{eqnarray}\label{J}
J=\frac{\int_{\varphi(U)}  e^{ s\hat{f}}\hat{c}\phi_R^2\sqrt{\det g_{ij}} \mathrm{d}x  }{\int_{\varphi(U)}  e^{ s\hat{f}} \phi_R^2\sqrt{\det g_{ij}}\mathrm{d}x  }\le \max_{\overline{B_{3R}(\varphi(p^*))}}\hat{c}.
 \end{eqnarray}

 Therefore, we complete the estimate and obtain
\begin{eqnarray}\label{IJ}
\lambda(s)\le I+J\le \frac{C|B_{3R}|    }{R^2|B_{R}|     }e^{-3sR^2}+ \max_{\overline{B_{3R}(\varphi(p^*))}}\hat{c}.
 \end{eqnarray}
 Sending $s\to+\infty$ first, it is easily seen that the first term of the right hand of \eqref{IJ} tends to $0$. Since $\varphi$ is a diffeomorphism, letting $R\to 0$ yields that
 $$
\limsup_{s\to\infty} \lambda(s)\le \hat{c}(\varphi(p^*))=c\circ \varphi^{-1}(\varphi(p^*))=c(p^*).
 $$
 Note that this estimating is valid for any maximum points of the Morse function $f$, and then we obtain that
 $$
 \limsup_{s\to\infty}\lambda(s)\le \min_{\mathcal{M}}c(p).
 $$
\end{proof}

\subsection{Lower bound}
Recall in Section \ref{pre}, we denote by $v(s,\cdot)=e^{-\frac{sf}{2}}w(s,\cdot)$ as the positive principal eigenfunction corresponding to $\lambda(s)$, satisfying
$$
\int_{M} e^{sf }v^2\mathrm{dV_g}=\int_{M}  w^2\mathrm{dV_g}=1.
$$
By Prokhorov's Theorem \cite{yan}, we could select a subsequence $\{s_m\}_{m=1}^{+\infty}$ such that
$$
\lim_{m\to+\infty}s_m=+\infty, \text{ and } \lim_{m\to+\infty}\lambda(s_m)=\liminf_{s\to+\infty}\lambda(s)\triangleq\lambda_*,\ \lim_{m\to+\infty}w(s_m,\cdot)=\mu\ \text{in measure}.
$$

Recall that we assume $c^*>c_*>0$. From the variational characterization and the boundedness of $\lambda(s)$, we have the following useful inequality
\begin{eqnarray}\label{key}
c^*-c_*\ge  \int_{M} e^{sf }|\nabla_M v|^2\mathrm{dV_g},
\end{eqnarray}
which will be used several times later.

To present the lower bound estimate of $\lambda(s)$, we divide it into two cases: the non-critical points and  the critical points which are not local maxima.

\subsubsection{The non-critical points}
We denote the non-critical points of $f$ by the set
\[
M_1 = \left\{ p \in M \,\left|\,
\begin{aligned}
& \text{For any local coordinate chart } (U, \varphi), \text{ the Riemannian} \\
& \text{metric } g_{ij} \text{ satisfies }
 g^{ij} \frac{\partial \hat{f}}{\partial x^i} \frac{\partial \hat{f}}{\partial x^j} > \delta \text{ for }p\in U \text{ i.e., } x \in \varphi(U)
\end{aligned} \right.
\right\}
\]
where $\delta$ is a positive constant. Then we have the following lemma.
\begin{lemma}\label{m1}
$\mu(M_1)=0$.
\end{lemma}
\begin{proof}
Fix an arbitrary point   $p^*\in M_1$.
Define a smooth cut-off function $\eta:\mathbb{R}^n\rightarrow \mathbb{R}$ such that
 $$
 \eta=1,\ \text{in } B_1(0), \quad \eta=0,\ \text{in }\mathbb{R}^n\setminus B_2(0), \quad 0\le\eta\le 1,\ \text{and } |\nabla\eta|\le C \text{ in }B_2(0).
 $$
Choose a sufficiently small $R$ such that $B_{3R}(\varphi(p^*))\subset\subset \varphi(U)$. Define the rescaled cut-off function $\eta_{R}=\eta(\frac{x-\varphi(p^*)}{R})$, which satisfies $0\leq\eta_{R}\leq 1$ and $|\nabla \eta_{R}|\leq \frac{C}{R}$ for $x\in\varphi(U)$.
We then define a smooth cut-off function $\xi$ on the manifold $M$ by
 \begin{equation}
 \xi=\begin{cases}
 \eta_R\circ\varphi, & p\in U,\\
 0,&p\in M\setminus U,
 \end{cases}
 \end{equation}
which satisfies $|\eta_R\circ\varphi|\le 1$, where $\varphi $ is a diffeomorphism defined on the neighborhood $U$ of $p^*$ in $M$.

By using \eqref{key}, and the transformation $w=e^{\frac{s f}{2}}v$, we have the following estimate.
\begin{align}\label{delta}
c^*-c_* 
\ge & \int_{U}\xi^2 e^{sf }|\nabla_M v|^2\mathrm{d V_g}\nonumber\\
= & \int_{\varphi(U)}\eta_R^2 e^{s\hat{f} }g^{ij}\frac{\partial \hat{v}}{\partial x^i}\frac{\partial \hat{v}}{\partial x^j}\sqrt{\det{g_{ij}}}\mathrm{d} x\nonumber\\
=&\int_{\varphi(U)}\eta_R^2  g^{ij}\left(\frac{\partial \hat{w}}{\partial x^i}-\frac{s}{2}\hat{w}\frac{\partial \hat{f}}{\partial x^i}\right)\left(\frac{\partial \hat{w}}{\partial x^j}-\frac{s}{2}\hat{w}\frac{\partial \hat{f}}{\partial x^j}\right)\sqrt{\det{g_{ij}}}\mathrm{d}x\nonumber\\
=&\int_{\varphi(U)}\eta_R^2  g^{ij}\frac{\partial \hat{w}}{\partial x^i}\frac{\partial \hat{w}}{\partial x^j}\sqrt{\det{g_{ij}}}\mathrm{d}x+\frac{s^2}{4}\int_{\varphi(U)}\eta_R^2  g^{ij}\hat{w}^2\frac{\partial \hat{f}}{\partial x^i}\frac{\partial \hat{f}}{\partial x^j}\sqrt{\det{g_{ij}}}\mathrm{d}x\nonumber\\
&-s\int_{\varphi(U)} \eta_R^2g^{ij}\hat{w}\frac{\partial \hat{w}}{\partial x^j}\frac{\partial \hat{f}}{\partial x^i}\sqrt{\det{g_{ij}}} \mathrm{d}x\nonumber\\
=&\int_{\varphi(U)}\eta_R^2  g^{ij}\frac{\partial \hat{w}}{\partial x^i}\frac{\partial \hat{w}}{\partial x^j}\sqrt{\det{g_{ij}}}\mathrm{d}x+\frac{s^2}{4}\int_{\varphi(U)}\eta_R^2  g^{ij}\hat{w}^2\frac{\partial \hat{f}}{\partial x^i}\frac{\partial \hat{f}}{\partial x^j}\sqrt{\det{g_{ij}}}\mathrm{d}x\nonumber\\
&+\frac{s}{2}\int_{\varphi(U)}\hat{w}^2\frac{\partial }{\partial x^j}\left(\eta_R^2g^{ij}\frac{\partial \hat{f}}{\partial x^i}\sqrt{\det{g_{ij}}}\right)\mathrm{d}x\nonumber\\
\ge&\frac{s^2}{4}\int_{\varphi(U)} \eta_R^2  g^{ij}\hat{w}^2\frac{\partial \hat{f}}{\partial x^i}\frac{\partial \hat{f}}{\partial x^j}\sqrt{\det{g_{ij}}}\mathrm{d}x +\frac{s}{2}\int_{\varphi(U)}\hat{w}^2\frac{\partial }{\partial x^j}\left(\eta_R^2g^{ij}\frac{\partial \hat{f}}{\partial x^i}\sqrt{\det{g_{ij}}}\right)\mathrm{d}x\nonumber\\
\triangleq& K_1+K_2,
\end{align}
where $\hat{(\cdot)}=(\cdot)\circ \varphi^{-1}$, the forth equality follows from integration by parts, and the final inequality uses the fact that
$$
g^{ij}\frac{\partial \hat{w}}{\partial x^i}\frac{\partial \hat{w}}{\partial x^j}\ge 0.
$$
 Here
 $$
 K_1=\frac{s^2}{4}\int_{\varphi(U)} \eta_R^2  g^{ij}\hat{w}^2\frac{\partial \hat{f}}{\partial x^i}\frac{\partial \hat{f}}{\partial x^j}\sqrt{\det{g_{ij}}}\mathrm{d}x,\ K_2=\frac{s}{2}\int_{\varphi(U)}\hat{w}^2\frac{\partial }{\partial x^j}\left(\eta_R^2g^{ij}\frac{\partial \hat{f}}{\partial x^i}\sqrt{\det{g_{ij}}}\right)\mathrm{d}x.
 $$

Note that the following inequality can be easily verified using Young's inequality:
\begin{align}\label{Y}
&s\int_{\varphi(U)}\left|\hat{w}^2\eta_R g^{ij}\frac{\partial \hat{f}}{\partial x^i}\frac{\partial \eta_R}{\partial x^j}\sqrt{\det{g_{ij}}} \right|\mathrm{d}x\nonumber\\
\le& \frac{ s^2}{8}\int_{\varphi(U)}\hat{w}^2\eta_R^2 g^{ij}\frac{\partial \hat{f}}{\partial x^i}\frac{\partial \hat{f}}{\partial x^j}\sqrt{\det{g_{ij}}} \mathrm{d}x+2\int_{\varphi(U)}\hat{w}^2g^{ij}\frac{\partial \eta_R}{\partial x^i}\frac{\partial \eta_R}{\partial x^j}\sqrt{\det{g_{ij}}} \mathrm{d}x.
\end{align}
There exists a positive constant $C$ such that
\begin{equation}\label{CDET}
\left|\frac{\partial }{\partial x^j}\left( g^{ij}\frac{\partial \hat{f}}{\partial x^i}\sqrt{\det{g_{ij}}}\right)\right|\le C\sqrt{\det g_{ij}},
\end{equation}
which follows the positivity of $\det{g_{ij}}$ and the smoothness of the functions involved.

Combining with \eqref{Y} and \eqref{CDET}, we estimate $K_2$ and  obtain
\begin{align*}
|K_2|\le&\frac{s}{2}\int_{\varphi(U)}\left|\hat{w}^2\frac{\partial }{\partial x^j}\left(\eta_R^2g^{ij}\frac{\partial \hat{f}}{\partial x^i}\sqrt{\det{g_{ij}}}\right)\right|\mathrm{d}x\nonumber\\
=&s\int_{\varphi(U)}\!\left|\hat{w}^2\eta_R g^{ij}\frac{\partial \hat{f}}{\partial x^i}\frac{\partial \eta_R}{\partial x^j}\sqrt{\det{g_{ij}}} \right|\mathrm{d}x
+\frac{s}{2}\int_{\varphi(U)}\!\hat{w}^2\eta_R^2\left|\frac{\partial }{\partial x^j}\left( g^{ij}\frac{\partial \hat{f}}{\partial x^i}\sqrt{\det{g_{ij}}}\right)\right|\mathrm{d}x \nonumber\\
\le& \frac{ s^2}{8}\int_{\varphi(U)}\hat{w}^2\eta_R^2 g^{ij}\frac{\partial \hat{f}}{\partial x^i}\frac{\partial \hat{f}}{\partial x^j}\sqrt{\det{g_{ij}}} \mathrm{d}x+2\int_{\varphi(U)}\hat{w}^2g^{ij}\frac{\partial \eta_R}{\partial x^i}\frac{\partial \eta_R}{\partial x^j}\sqrt{\det{g_{ij}}} \mathrm{d}x \\
&+ Cs \int_{\varphi(U)} \hat{w}^2\eta_R^2 \sqrt{\det g_{ij}} \mathrm{d}x\nonumber\\
\le& \frac{ s^2}{8}\int_{\varphi(U)}\hat{w}^2\eta_R^2 g^{ij}\frac{\partial \hat{f}}{\partial x^i}\frac{\partial \hat{f}}{\partial x^j}\sqrt{\det{g_{ij}}} \mathrm{d}x+\frac{C}{R^2}\int_{\varphi(U)}\hat{w}^2\sqrt{\det g_{ij}} \mathrm{d}x
\\
&+Cs\int_{\varphi(U)} \hat{w}^2\eta_R^2 \sqrt{\det g_{ij}} \mathrm{d}x,
\end{align*}
where we use the fact that
$$
g^{ij}\frac{\partial  \eta_R}{\partial x^i}\frac{\partial  \eta_R}{\partial x^j}\le C |\nabla \eta_R|^2\le \frac{C}{R^2}
$$
for some constant $C>0$
in the last inequality.

Recalling \eqref{delta} and the definition of the set $M_1$,  we get
\begin{align}\label{K}
c^*-c_*
\ge & K_1+K_2
\ge K_1-|K_2|\nonumber\\
\ge& \frac{s^2 }{8} \int_{\varphi(U)} \eta_R^2 \hat{w}^2 g^{ij}\frac{\partial \hat{f}}{\partial x^i}\frac{\partial \hat{f}}{\partial x^j}\sqrt{\det{g_{ij}}}\mathrm{d}x-\frac{C}{R^2}\int_{\varphi(U)}\hat{w}^2\sqrt{\det g_{ij}} \mathrm{d}x
\nonumber\\
-&Cs\int_{\varphi(U)} \eta_R^2\hat{w}^2 \sqrt{\det g_{ij}} \mathrm{d}x\nonumber\\
\ge& \frac{s( \delta s-C)}{8} \int_{\varphi(U)} \eta_R^2 \hat{w}^2 \sqrt{\det{g_{ij}}}\mathrm{d}x-\frac{C}{R^2}\int_{U}e^{sf}v^2\mathrm{d V_g} \nonumber\\
\ge& \frac{s( \delta s-C)}{8} \int_{\varphi(U)} \eta_R^2 \hat{w}^2 \sqrt{\det{g_{ij}}}\mathrm{d}x-\frac{C}{R^2}.
\end{align}
Hence, in view of \eqref{K} and \eqref{weak},  we obtain that
\begin{align*}
0
&=\lim_{m\to+\infty}\int_{\varphi(U)} \eta_R^2 \hat{w}(s_m,x)^2 \sqrt{\det{g_{ij}}}\mathrm{d}x\nonumber\\
&=\lim_{m\to+\infty}\int_{M} \xi^2 w^2(s_m,p) \mathrm{d V_g}\nonumber\\
&=\int_M \xi^2\mu(\mathrm{d V_g})\nonumber\\
&\ge\mu(\varphi^{-1}(B_{\frac{R}{2}}(\varphi(p^*)))),
\end{align*}
according to $\xi=1$ in $\varphi^{-1}(B_{\frac{R}{2}}(\varphi(p^*)))$. Therefore we conclude that $\mu(M_1)=0$.
\end{proof}

\subsubsection{Critical points which are not local maxima}

Recalling Lemma \ref{morsel}, we denote the  critical points which are not local maxima by the set
\begin{equation}
M_2 \!=\! \left\{ p \in M \!\left|\,
\begin{aligned}
& df_p = 0, \text{ there exists a local coordinate chart } (U, \varphi)\text{ with}\\
&\text{Riemannian metric } g_{ij} \text{ such that} \\
& f \circ \varphi^{-1}(x) \!=\! f(p) - \frac{1}{2} \sum_{i=1}^{m_1} \!(x^i)^2 + \frac{1}{2} \sum_{i=m_1+1}^{n} \!(x^i)^2,\ x \in \varphi(U),\ 0\leq m_1 < n
\end{aligned} \right.
\right\}
\end{equation}
Based on the definition of $M_2$, we have the following lemma.
\begin{lemma}\label{m2}
$\mu(M_2)=0$.
\end{lemma}

\begin{proof}
Fix an arbitrary point $p^*\in M_2$.
By the definition of set $M_2$, there exists an local coordinate $(U, \varphi)$ and an integer  $l\in\{m_1+1,...,n\}$ such that
$$
 \frac{\partial^2( f\circ \varphi^{-1})}{\partial (x^l)^2}=\frac{\partial^2 \hat{f}}{\partial (x^l)^2}=1,\quad x\in \varphi(U).
$$

Set $\eta_R=\eta(\frac{x-\varphi(p^*)}{R})$ where $\eta_R$ is the cut-off function defined in the proof of Lemma \ref{m1} with
\begin{equation}\label{bound}
|\eta_R|\le 1,\quad
\left(\frac{\partial \eta_R}{\partial x^l}\right)^2\le|\nabla \eta_R|^2\le \frac{C}{R^2} \text{ for } x\in \varphi(U).
\end{equation}

Prior to presenting the proof, we first establish several inequalities that will facilitate subsequent analysis.
In view of the Young's inequality and \eqref{bound}, we obtain the following inequalities
\begin{align}\label{NM1}
&\left|\int_{\psi(U)}s\hat{w}^2\eta_R \frac{\partial \eta_R }{\partial x^l}\frac{\partial \hat{f}}{\partial x^l} \sqrt{\det{g_{ij}}}\mathrm{d}x\right|\nonumber\\
\le&\epsilon s^2\int_{\psi(U)}\hat{w}^2\eta_R^2\left(\frac{\partial \hat{f}}{\partial x^l}\right)^2\sqrt{\det g_{ij}}\mathrm{d}x+\frac{C}{\epsilon}\int_{\psi(U)}\hat{w}^2\left(\frac{\partial \eta_R}{\partial x^l}\right)^2\sqrt{\det g_{ij}}\mathrm{d}x\nonumber\\
\le&\epsilon s^2\int_{\psi(U)}\hat{w}^2\eta_R^2 \left(\frac{\partial \hat{f}}{\partial x^l}\right)^2\sqrt{\det g_{ij}}\mathrm{d}x+\frac{C}{\epsilon R^2}\int_{\psi(U)}\hat{w}^2\sqrt{\det g_{ij}}\mathrm{d}x\nonumber\\
=&\epsilon s^2\int_{\psi(U)}\hat{w}^2\eta_R^2 \left(\frac{\partial \hat{f}}{\partial x^l}\right)^2\sqrt{\det g_{ij}}\mathrm{d}x+\frac{C}{\epsilon R^2}\int_{U}w^2\mathrm{d V_g}\nonumber\\
\le&\epsilon s^2\int_{\psi(U)}\hat{w}^2 \eta_R^2\left(\frac{\partial \hat{f}}{\partial x^l}\right)^2\sqrt{\det g_{ij}}\mathrm{d}x+\frac{C}{\epsilon R^2},
\end{align}
where $\epsilon$ is a positive constant that will be determined later.

Thanks again to the Young's inequality and the positivity of the metric $g$, we obtain
\begin{align}\label{NM2}
&\frac{s}{2}\left|\int_{\psi(U)}\hat{w}^2\eta_R^2\frac{\partial \hat{f}}{\partial x^l}\frac{\partial \sqrt{\det{g_{ij}}}}{\partial x^l}\mathrm{d}x\right|\nonumber\\
=&\left|\int_{\psi(U)}\left(\frac{s}{2}\hat{w}\eta_R\frac{\partial \hat{f}}{\partial x^l}(\det g_{ij})^{\frac{1}{4}}\right)\left(\hat{w}\eta_R(\det g_{ij})^{-\frac{1}{4}}\frac{\partial \sqrt{\det{g_{ij}}}}{\partial x^l}\right)\mathrm{d}x\right|\nonumber\\
\le&\frac{\epsilon s^2}{4}\int_{\psi(U)}\hat{w}^2\eta_R^2\left(\frac{\partial \hat{f}}{\partial x^l}\right)^2\sqrt{\det g_{ij}}\mathrm{d}x+\frac{C}{\epsilon}\int_{\psi(U)}\hat{w}^2\eta_R^2(\det g_{ij})^{-\frac{1}{2}}\left(\frac{\partial \sqrt{\det{g_{ij}}}}{\partial x^l}\right)^2\mathrm{d}x\nonumber\\
\le&\frac{\epsilon s^2}{4}\int_{\psi(U)}\hat{w}^2\eta_R^2\left(\frac{\partial \hat{f}}{\partial x^l}\right)^2\sqrt{\det g_{ij}}\mathrm{d}x+\frac{C}{\epsilon}\int_{U}w^2\mathrm{d V_g}\nonumber\\
\le&\frac{\epsilon s^2}{4}\int_{\psi(U)}\hat{w}^2\eta_R^2\left(\frac{\partial \hat{f}}{\partial x^l}\right)^2\sqrt{\det g_{ij}}\mathrm{d}x+\frac{C}{\epsilon}.
\end{align}
By using \eqref{key}, and the transformation $\hat{w}=e^{\frac{s\hat{f}}{2}}\hat{u}$, we have the following estimate:
\begin{align}\label{est2}
&c^*-c_* \nonumber\\
\ge&\int_{\psi(U)}\eta_R^2  g^{ij}\left(\frac{\partial \hat{w}}{\partial x^i}-\frac{s}{2}\hat{w}\frac{\partial \hat{f}}{\partial x^i}\right)\left(\frac{\partial \hat{w}}{\partial x^j}-\frac{s}{2}\hat{w}\frac{\partial \hat{f}}{\partial x^j}\right)\sqrt{\det{g_{ij}}}\mathrm{d}x\nonumber\\
\ge &C\int_{\psi(U)}\eta_R^2  \left|\nabla \hat{w}-\frac{s}{2}\hat{w}\nabla \hat{f}\right|^2\sqrt{\det{g_{ij}}}\mathrm{d}x\nonumber\\
\ge& C\int_{\psi(U)}\eta_R^2  \left|  \frac{\partial \hat{w}}{\partial x^l}-\frac{s}{2}\hat{w}\frac{\partial \hat{f}}{\partial x^l} \right|^2\sqrt{\det{g_{ij}}}\mathrm{d}x\nonumber\\
=& C\int_{\psi(U)}\eta_R^2\left[\left(\frac{\partial \hat{w}}{\partial x^l}\right)^2-s\hat{w}\frac{\partial \hat{f}}{\partial x^l}  \frac{\partial \hat{w}}{\partial x^l}+  \frac{s^2}{4}\hat{w}^2\left(\frac{\partial \hat{f}}{\partial x^l} \right)^2\right]\sqrt{\det{g_{ij}}}\mathrm{d}x\nonumber\\
\ge& C\left[\frac{s}{2}\int_{\psi(U)}\!\!\hat{w}^2 \frac{\partial  }{\partial x^l}\left(\eta_R^2\frac{\partial \hat{f}}{\partial x^l} \sqrt{\det{g_{ij}}}\right)\mathrm{d}x+\frac{s^2}{4}\int_{\psi(U)}\eta_R^2 \hat{w}^2\left(\frac{\partial \hat{f}}{\partial x^l} \right)^2\sqrt{\det{g_{ij}}}\mathrm{d}x\right] \nonumber\\
=& C \left[s\int_{\psi(U)}\hat{w}^2\eta_R \frac{\partial \eta_R }{\partial x^l}\frac{\partial \hat{f}}{\partial x^l} \sqrt{\det{g_{ij}}}\mathrm{d}x+\frac{s}{2}\int_{\psi(U)}\hat{w}^2\eta_R^2 \frac{\partial^2 \hat{f}}{\partial (x^l)^2}\sqrt{\det{g_{ij}}}\mathrm{d}x\right.\nonumber\\
& \left.+\frac{s}{2}\int_{\psi(U)}\hat{w}^2\eta_R^2\frac{\partial \hat{f}}{\partial x^l}\frac{\partial \sqrt{\det{g_{ij}}}}{\partial x^l}\mathrm{d}x+\frac{s^2}{4}\int_{\psi(U)}\hat{w}^2\eta_R^2\left(\frac{\partial \hat{f}}{\partial x^l} \right)^2\sqrt{\det{g_{ij}}}\mathrm{d}x \right] \nonumber\\
\ge& C\left[\!\frac{(1\!-\!5\epsilon)s^2}{4}\!\!\int_{\psi(U)}\!\!\hat{w}^2\eta_R^2\left(\frac{\partial \hat{f}}{\partial x^l} \right)^2\!\!\!\sqrt{\det{g_{ij}}}\mathrm{d}x\!-\!\frac{C}{\epsilon}(1\!+\!\frac{1}{R^2})\!+\!\frac{s}{2}\!\int_{\psi(U)}\!\!\hat{w}^2\!\eta_R^2 \! \frac{\partial^2 \hat{f}}{\partial (x^l)^2}\!\sqrt{\det{g_{ij}}}\mathrm{d}x\!\right]\nonumber\\
\ge& 
Cs\int_{\psi(U)}\hat{w}^2\eta_R^2 \frac{\partial^2 \hat{f}}{\partial (x^l)^2}\sqrt{\det{g_{ij}}}\mathrm{d}x-C(1+\frac{1}{R^2})\nonumber\\
\ge& Cs\int_{\psi(U)}\hat{w}^2\eta_R^2  \sqrt{\det{g_{ij}}}\mathrm{d}x-\frac{C}{R^2}.
\end{align}
In the second inequality and the forth inequality, we use the estimates \eqref{g12} and integral by parts, respectively. In the fifth inequality, we employ \eqref{NM1} and \eqref{NM2}. Selecting $\epsilon=\frac{1}{6}$ yields the subsequent inequality, and the final step utilize the identity $\frac{\partial^2 \hat{f}}{\partial (x^l)^2}=1$ along with the fact that the positive constant $R$ is sufficiently small.

Similarly as the proof in Lemma \ref{m1}, in view of \eqref{est2} and \eqref{weak},
we conclude that $\mu(M_2)=0$.
\end{proof}
\subsubsection{Proof of Theorem \ref{mmm}}
According to Lemma \ref{up}, Lemma \ref{m1} and Lemma \ref{m2} and the variational characterization \eqref{voe}, we obtain
$$
\lambda_*\!\ge\!
\lim_{m\to+\infty}\int_M\! c(p)w^2(s_m,p)\mathrm{dV_g}\!
=\!\!\int_{M}\! c(p)\mu(\mathrm{dV_g})\!
=\!\!\int_{\mathcal{M}} c(p)\mu(\mathrm{dV_g})
\!\ge\!\min_{\mathcal{M}}c(p)
\!\ge\! \lambda^*\!\ge\!\lambda_*,
$$
where $w(s,\cdot)=e^{\frac{s f}{2}}v(s,\cdot)$ for the normalized eigenfunction $v(s,\cdot)$, and $\lambda^*=\limsup\limits_{s\to\infty}\lambda(s)$, Then we complete the proof.

\section{Appendix}\label{app}
In this section, we give the proof of the variational characterization of the principal eigenvalue $\lambda(s)$ of the problem \eqref{E}. Without loss of generality, we assume that $c(p)\ge c_*>0$.

To proceed, we first place the eigenvalue problem in the setting of a smooth metric measure space  $(M, g, e^{sf}\mathrm{d} V_g)$ (see \cite{zhoudetang}).  It consists of a closed orientable \(n\)-dimensional Riemannian manifold \((M, g)\) together with a weighted volume form \(e^{sf}\mathrm{d} V_g\), where $\mathrm{d}V_g$ is the volume form (Riemannian volume element) on    \((M,g)\).
For convenience, we denote by $\mu$ the
measure induced by the weighted volume element \(e^{sf}\mathrm{d} V_g\), i.e. $\mathrm{d}\mu=e^{sf} \mathrm{d} V_g$.
Let $H^1(M,\mu)$ denote the  space of functions in $L^2(M,\mu)$ whose gradient
 is square-integral with respect to the measure $\mu$. $L^2(M,\mu)$ and $H^1(M,\mu)$ is the Hilbert space with the following  norms:
 \begin{eqnarray*}
 \|u\|_{L^2(M,\mu)}=\left(\int_M u^2\mathrm{d} \mu\right)^{\frac{1}{2}},\quad
 \|u\|_{H^1(M,\mu)}=\left(\int_M (u^2+|\nabla_M u|^2)\mathrm{d} \mu\right)^{\frac{1}{2}}.
 \end{eqnarray*}

 Let $$J(u)=\frac{\int_{M} \left(|\nabla_M u|^2+c |u|^2\right)\mathrm{d}\mu}{\int_{M} |u|^2\mathrm{d}\mu}\ \left(=\frac{\int_{M} \left(|\nabla_M u|^2+c |u|^2\right)e^{sf}\mathrm{d}V_g}{\int_{M} |u|^2e^{sf}\mathrm{d}V_g}\right)$$
 with $\mathrm{d}\mu=e^{sf}\mathrm{d}V_g$. Then, we obtain the variational characterization of $\lambda(s)$ of the problem \eqref{E}.
 \begin{lemma}
 For  $(M, g, \mathrm{d} \mu)$ $(=(M, g, e^{sf}\mathrm{d} V_g))$, the principal eigenvalue $\lambda(s)$ of problem \eqref{E} can be characterized by
 \begin{equation}\label{cha}
 \lambda(s)=\min_{u\in H^1(M,\mu)\atop u\not\equiv 0} J(u).
 \end{equation}
 \end{lemma}

 \begin{proof}
 From the definition of $J(u)$, we note that $J(u)$ is bounded below. So the infimum of the functional $J(u)$ is well defined. Define $$\lambda(s)=\inf\limits_{u\in H^1(M,\mu)\atop u\not\equiv 0} J(u),$$ then we need to prove that $J(u)$ achieves its minimum at the space $H^1(M,\mu)\setminus\{0\}$ and $\lambda(s)$ is the principal eigenvalue of problem \eqref{E}, i.e. there exists a non-trivial function $v\in H^1(M,\mu)\setminus\{0\}$ such that $\inf\limits_{u\in H^1(M,\mu)\atop u\not\equiv 0} J(u)=J(v)$ and $L_M v=\lambda(s)v$ on $M$.

 To this end, we select a minimizing function sequence  $\{u_m\}_{m=1}^{+\infty}$ in $H^1(M,\mu)$ satisfying
 \begin{equation}\label{mod}
 \int_M u_m^2\mathrm{d} \mu\equiv 1,
 \end{equation}
 such that
$$
\lambda(s)=\lim_{m\to+\infty}J(u_m)=\inf_{u\in H^1(M,\mu)\atop u\not\equiv 0} J(u).
$$
Then there exists a constant $C(s)$, independent of $m$, such that
\begin{equation}\label{mos}
C(s)\ge \int_{M} \left(|\nabla_M u_m|^2+c |u_m|^2\right)\mathrm{d}\mu\ge \int_{M} |\nabla_M u_m|^2\mathrm{d}\mu, \, m=1,2,...
\end{equation}
Combining \eqref{mod} with \eqref{mos}, we deduct that the sequence $\{u_m\}_{m=1}^{+\infty}$ is uniformly bounded in $H^1(M,\mu)$. Thus, there exists a subsequence  $\{u_{m_k}\}_{k=1}^{+\infty}$ of $\{u_m\}_{m=1}^{+\infty}$ such that $u_{m_k}$ weakly converge to some $v$ in $H^1(M,\mu)$. Hence, noting that $c\ge c_*>0$ is smooth on $M$, according to the weak lower semi-continuity of the $L^2$ norm, we obtain
$$
\liminf_{k\to+\infty}J(u_{m_k})=\liminf_{k\to+\infty}\int_{M} \left(|\nabla_M u_{m_k}|^2+c |u_{m_k}|^2\right)\mathrm{d}\mu\ge \int_{M} \left(|\nabla_Mv|^2+c |v|^2\right)\mathrm{d}\mu=J(v).
$$
 We derive that
$$\lambda(s)=\inf_{u\in H^1(M,\mu)\atop u\not\equiv 0} J(u)\ge J(v)\ge \lambda(s),$$
which implies that $\lambda(s)=\min\limits_{u\in H^1(M,\mu)\atop u\not\equiv 0} J(u)=J(v)$.

We still need to present that $v$ satisfies the eigenvalue problem \eqref{E} in weak sense. Let $h(t)=J(v+t\psi)$ for $\psi\in H^1(M,\mu)$.
Direct calculations yields that
\begin{align*}
0=&h'(0)\int_M v^2\mathrm{d}\mu \nonumber\\
=&2\int_M (\nabla_M v,\nabla_M\psi)_g+(c-\lambda(s))v\psi\mathrm{d}\mu\nonumber\\
=&2\int_M\left[(\nabla_M v,\nabla_M\psi)_g+(c-\lambda(s))v\psi\right]e^{sf}\mathrm{d}V_g\nonumber\\
=&2\int_M \dive(\psi e^{sf}\nabla_M v)\mathrm{d}V_g-2\int_M[(\Delta_M v)+s(\nabla_M v, \nabla_M f )_g-(c-\lambda(s))v]\psi e^{sf}\mathrm{d}V_g\nonumber\\
=&-2\int_M[(\Delta_M v)+s(\nabla_M v, \nabla_M f )_g-(c-\lambda(s))v]\psi e^{sf}\mathrm{d}V_g,\quad \text{for all } \psi\in H^1(M,\mu).
\end{align*}
The third equality follows from $\mathrm{d}\mu=e^{sf}\mathrm{d}V_g$. The forth and fifth  equalities follows from the divergence theorem (see \cite{jmlee}). We now complete the proof.
 \end{proof}

\bigbreak
\noindent \textbf{Acknowledgment}.
The research of Xin Xu is supported by NSF of China (Youth Program No. 12401255). 


\bigbreak
\noindent\textbf{Declaration of competing interest}
The authors declare that they have no known competing financial interests or personal relationships that could have appeared to influence the work reported in this paper.

\bigbreak
\noindent\textbf{Data availability statement}
Data sharing is not applicable to this article as no new data were created or analysed in this study.

\bibliographystyle{abbrv}
\bibliography{bib}

\end{document}